\documentclass[12pt, reqno]{amsart} 
\usepackage{graphicx}
\usepackage[margin=1.25in]{geometry}
\usepackage{amsmath, amssymb, amsthm}
\usepackage[colorlinks,linkcolor=blue]{hyperref}
\usepackage{mathtools}
\usepackage{bbm}
\usepackage{graphicx, xcolor}
\usepackage{subcaption}
\usepackage{float}
\usepackage{fullpage}

\usepackage[colorinlistoftodos]{todonotes}

\usepackage[backend=biber,sorting=nty,style=numeric,maxbibnames=99]{biblatex}
\def\RR{{\mathbb R}}
\def\SS{{\mathbb S}}

\def\l{\lambda}

\let\temp\phi
\let\phi\varphi
\let\varphi\temp

\def\ones{\mathbbm{1}}
\def\zeros{\mathbf{0} }

\def\supp{\operatorname {supp}}

\def\spanset{\operatorname {span}}

\def\rank{\operatorname {rank}}

\theoremstyle{definition}
\newtheorem{definition}{Definition}[section]

\theoremstyle{plain}
\newtheorem{theorem}[definition]{Theorem}
\newtheorem{lemma}[definition]{Lemma}
\newtheorem{proposition}[definition]{Proposition}
\newtheorem{corollary}[definition]{Corollary}
\newtheorem*{corollary*}{Corollary}

\newtheorem{remark}[definition]{Remark}

\DeclareMathOperator*{\minimize}{minimize}

\begin{document}

\title{Sparse Graphical Designs via Linear Programming}
\date{September 2023}

\author[Al-Thani]{Hessa Al-Thani}
\address[H. Al-Thani]{Industrial and Operations Engineering Department,
University of Michigan, Ann Arbor}
\email{\textcolor{blue}{\href{mailto:}{hessakh@umich.edu}}}

\author[Babecki]{Catherine Babecki}
\address[C. Babecki]{Department of Mathematics and the Computing + Mathematical Sciences, California Institute of Technology}
\email{\textcolor{blue}{\href{mailto:}{cbabecki@caltech.edu}}}

\author[Mart\'inez Mori]{J. Carlos Mart\'{i}nez Mori}
\address[J.~C. Mart\'inez Mori]{Schmidt Science Fellows}
\email{\textcolor{blue}{\href{mailto:jmartinezmori@schmidtsciencefellows.org}{jmartinezmori@schmidtsciencefellows.org}}}

\maketitle

\begin{abstract}
Graphical designs are a framework for sampling and numerical integration of functions on graphs.
In this note, we introduce a method to address the trade-off between graphical design sparsity and accuracy.
We show how to obtain sparse graphical designs via linear programming and design objective functions that aim to maximize their accuracy.
We showcase our approach using yellow taxicab data from New York City.
\end{abstract}

\section{Introduction}
\label{sec: introduction}

Graphs are a useful modeling tool in a variety of domains such as cyber-physical systems, the social sciences, and epidemiology.
Applications in these areas often require large data sets to be collected and monitored over time.
Their complexity and sheer scale generally turns this into a challenging task, which motivates the field of \emph{graph signal processing}~\cite{ortegaBook}.

\emph{Graphical designs}~\cite{SteinerbergerGraphDesigns} are a nascent research area at the intersection of graph signal processing, combinatorics, and optimization that provide a framework for sampling and numerically integrating functions on (the nodes of) a graph.
The goal behind numerical integration is to capture the \emph{global} behavior of a function by observing only \emph{local} information.
A graphical design is a possibly weighted subset of nodes that captures the global behavior of a given family of functions on a graph.

Naturally, graphical designs exhibit a trade-off between sparsity and accuracy.
That is, the more nodes one is allowed to monitor, the larger the family of functions one is able to capture exactly (and the better job one can do with functions outside the family).
In this work, we introduce a method to address this trade-off.
Our contributions are two-fold:
\begin{enumerate}
    \item 
    In Theorem~\ref{theorem: sparse}, we show that sparse graphical designs can be obtained via linear programming, where the sparsity guarantee depends on the dimension of the space of functions one seeks to capture exactly.
    This result follows readily from structural properties of extreme point solutions.
    \item
    The use of linear programming provides additional modeling flexibility that can help distinguish otherwise equivalent graphical designs.
    That is, based on deployment goals, one can design objective functions that assign different ``costs'' to different solutions.
    In Proposition~\ref{proposition: parametric bound} and Corollary~\ref{corollary: nonparametric bound}, we design objective functions that aim to capture (as accurately as possible) functions different from those already captured by Theorem~\ref{theorem: sparse}, in this way promoting robust numerical integration.
\end{enumerate}

We showcase our approach to graphical design optimization using yellow taxicab data from New York City (NYC).
Our computational experiments suggest that, in this setting, monitoring only a small fraction of the nodes ($\sim 5\%$) suffices to consistently capture the global behavior of historical data with under $5\%$ error.

The remainder of this note is organized as follows.
Section~\ref{sec: background} provides some background on quadrature rules, spectral graph theory, and graphical designs.
In Section~\ref{sec: graphical design optimization}, we develop our method for graphical design using linear programming.
Lastly, Section~\ref{sec: computational experiments} summarizes our computational experiments.

\section{Background}
\label{sec: background}

\subsection{Quadrature Rules and Sampling}
\label{sec: quadrature rules and sampling}

Graphical designs were introduced by Steinerberger~\cite{SteinerbergerGraphDesigns} as an analogue of classical quadrature rules on continuous domains to the realm of finite graphs. Essentially any book on approximation or numerical methods will contain an introduction to quadrature, for instance \cite{HammingNumericalMethodsBook, ApproxTheoryBook}.
In broad strokes, a \emph{quadrature rule} for a domain $\Omega$ is a set of points $x_1, x_2, \ldots, x_n\in \Omega$ together with weights $a_1, a_2, \ldots, a_n \in \RR$ so that if a function $f: \Omega \to \RR$ is ``nice enough,'' the average of $f$ over $\Omega$ is approximated by the weighted average over the quadrature points: 
\begin{equation}
\label{eq: general quadrature}
\frac{1}{\mu(\Omega)}\int_\Omega f(x) d\mu(x) \approx \sum_{i=1}^n a_i f(x_i).     
\end{equation} 

Different definitions of ``nice enough'' and ``good approximation'' give rise to different quadrature rules.
Graphical designs are most closely inspired by Sobolev-Lebedev quadrature \cite{LebedevQuad, SobolevQuad} and spherical $t$-designs \cite{DGSSphericalCodes}, which are quadrature rules for the sphere $\SS^{d-1}$ for which (\ref{eq: general quadrature}) holds at equality for the \emph{low frequency eigenfunctions} of the spherical Laplacian operator 
\begin{equation*}
    \Delta_{\SS^{d-1}} = \sum_{i=1}^d \frac{\partial^2}{\partial x_i^2}.
\end{equation*}
A function $f$ is an \emph{eigenfunction} of an operator $\Delta$ if there is $\l \in \RR$ so that $\Delta f = \l f$.  If $|\l|$ is small, $f$ has low frequency; if $|\l|$ is large, $f$ has high frequency.
The eigenfunctions of $\Delta_{\SS^2}$ are known as \emph{spherical harmonics}, and they play a similar role to $\sin k\theta$ and $\cos k\theta$ in the higher-dimensional version of Fourier series.  For further introduction to spherical harmonics, we refer to the survey \cite{sphericalSurvey}.
Low frequency spherical harmonics (think: $\sin \theta)$ are smoother with respect to the structure and symmetry of the underlying sphere, in the sense that in the neighborhood of a point, the function does not vary much. A subset of points that averages these functions intuitively captures this structure and symmetry.
On the other hand, high frequency eigenfunctions  (think: $\sin 500\theta)$ are highly oscillatory, and there is not much hope to capture their global behavior through only a few points.

Quadrature rules for very simple domains like line segments, boxes, and spheres are classical and even ancient, dating back to the ancient Greeks and Babylonians and scientists such as Gauss (Gaussian quadrature) and Newton (Newton-Coates formulas).
Numerical integration on general manifolds is more contemporary, 
see, for instance, \cite{BRVmanifoldsI,BRVmanifoldsII,GGManifoldDesigns, SteinerbergerRiemannDesigns}.
We also refer to the work of Pesenson~(e.g., \cite{PesensonIII,   PesensonII, PesensonI}).

Sampling on graphs has been investigated primarily from a graph signal processing (GSP) perspective.
Broadly speaking, GSP is concerned with observing signals (i.e., graph functions) on a subset of nodes to reconstruct unobserved signals on the remainder of the graph~\cite{ortegaBook}.
Graphical design can be seen as an area within GSP with the much more targeted goal of numerical integration.
Naturally, both GSP and graphical design require structural assumptions on the class of observed signals (e.g., \emph{bandlimited} signals).
However, since in practice signals may not exhibit the assumed structure (e.g., they may not be truly bandlimited), further design criteria for robust sampling are typically considered.
In the GSP literature, greedy or randomized sampling algorithms have been devised for the general goal of robust signal reconstruction~\cite{samplingAGO, samplingBWCNG,SamplingCVSK, samplingMSLR, samplingSA, samplingTEOC, samplingTBD}.
In comparison, in this note we show that linear programming can be used for sampling with the more targeted goal of numerical integration, and that robustness can be accounted for through the objective function.
However, we note that one drawback of this approach is that it requires spectral decomposition, which induces significant storage requirements for large graphs.

Throughout this paper, we let $G = (V,E,w)$ be a simple, finite, connected graph with positive edge weights $w: E \rightarrow 
\mathbb{R}_{> 0}$, where $V = [n] \coloneqq \{1, 2, \ldots, n\}$, and we consider functions $f: [n] \rightarrow \mathbb{R}$ on the nodes.
We often think of $f$ as a vector $f \in \mathbb{R}^n$, where its $i$th entry is $f(i)$. 
Translating (\ref{eq: general quadrature}) to functions on graphs, we say a quadrature rule on $G$ is a proper subset of the nodes $S \subset [n]$ and weights $a_i \in \RR$ for $i \in S$ so that 
\begin{equation}
\label{eq: graph quadrature} 
    \frac{1}{n}\sum_{i \in [n]} f(i) \approx \sum_{i \in S} a_i f(i)
\end{equation} 
for some choice of ``nice enough'' functions and of ``good approximation.''

\subsection{The Graph Laplacian}
\label{sec: laplacian matrix}

There are many senses in which graph Laplacians are an appropriate translation to graphs of the spherical Laplacian and the more general Laplace-Beltrami operator for a smooth manifold; see \cite{ BIKGraphLaplacians,HALGraphLaplacians, Singer}, for instance. 
Thus we look to the eigenfunctions of graph Laplacians, which are just the eigenvectors of a matrix, to provide the classes of functions we seek to approximate well through quadrature. The study of graph operators and how their spectral properties relate to the combinatorial structure of the graph is known as \emph{spectral graph theory} or \emph{algebraic graph theory}; we refer to \cite{ChungSpectral,GodsilRoyleBook, SpielmanBook}.

Let $A \in \mathbb{R}^{n \times n}$ be the weighted \emph{adjacency matrix of $G$}, where $A_{ij} = w(ij)$ if $ij \in E$ and $A_{ij} = 0$ otherwise.
Similarly, let $D \in \mathbb{R}^{n \times n}$ be the diagonal \emph{degree matrix} recording the weighted degree of the nodes:  $D_{ii} = \sum_{j = 1}^n  A_{ij} = \sum_{ij \in E} w(ij)$ and $D_{ij} = 0$ for $i \neq j$.
The \emph{combinatorial Laplacian} of $G$ is the matrix $L = D - A$. 
The matrix $L$ is positive semidefinite (PSD), denoted $L\succeq 0$, as seen from its quadratic form:
\begin{equation*}
   x^\top L x = \sum_{ij \in E} w_{ij} (x_i - x_j)^2 \geq 0.  
\end{equation*}

There are many graph Laplacians worth considering, including $I - AD^{-1}$ and $D^{-1/2} A D^{-1/2}$. We focus on $L = D-A$ because it has many desirable structural properties, some of which we highlight next.
Since $L\succeq 0$, it has non-negative eigenvalues $ 0 \leq \lambda_1 \leq \lambda_2 \leq \cdots \leq \lambda_n$ with corresponding eigenvectors $\phi_1, \phi_2, \ldots, \phi_n$ that form an orthogonal basis of $\mathbb{R}^n$. This ordering on the spectrum of $L$ is analogous to the frequency ordering in the spherical case.  
If $\l_i$ is small, $\phi_i$ is ``smoother'' with respect to the graph's geometry; if $\l_i$  is large, $\phi_i$ is highly oscillatory \cite{SteinerbergerEigenfunctions}.  
It is a standard fact of spectral graph theory that the all-ones vector $\ones_n$ spans the eigenspace for $\l_1 =0$ if and only if $G$ is connected.

\subsection{Graphical Designs}
\label{sec: graphical designs}

We first establish some notation.  We will often refer to the eigenvectors and eigenvalues of $L$ as the eigenvectors and eigenvalues of the graph $G$.  The \emph{support} of a vector $a\in \RR^n$ is $\supp(a) = \{ i \in [n] : a_i \neq 0\}$.
Let $S \subset [n]$ be a subset of the nodes and $a \in \mathbb{R}^n$ be supported on $S$. 
We say the pair $(S,a)$ \emph{averages} a function $f$ if 
\begin{equation}
\label{eq: globally}
   \frac{1}{n} \sum_{i \in [n]} f(i) =  \sum_{i \in S} a_i f(i).  
\end{equation}

To mimic quadrature rules for the sphere, graphical designs were first defined by averaging the low frequency eigenvectors of a graph. 

\begin{definition}[\cite{SteinerbergerGraphDesigns}]
Let $G = ([n],E,w)$ be a positively weighted, connected, simple graph with eigenvectors $\phi_1 ,\phi_2,\ldots, \phi_n$ ordered by frequency. 
A \emph{$k$-graphical design} is a pair $(S,a)$ that averages $\phi_1 ,\phi_2,\ldots, \phi_k$ simultaneously: 
\begin{align*}
     \frac{1}{n} \sum_{i \in [n]} \phi_j(i) = \sum_{i \in S} a_i \phi_j(i) \qquad \text{  for all  } j \in [k].
\end{align*}
\end{definition}

To be precise, this definition may be ill-defined if a graph has eigenspaces with multiplicity \cite{BabeckiGraphDesignsCodes}.
Eigenspace multiplicity hints that a graph has some additional structure and symmetry, the most extreme case being strongly regular graphs that have only three eigenspaces \cite[Lemma 10.2.1]{GodsilRoyleBook}.
Graphs arising from real-world data are unlikely to have such symmetries, thus the technicality about eigenspace multiplicity is unlikely to be significant in applications.
We also note that negative quadrature weights are typically undesirable, as they can lead to unstable or divergent solutions \cite{HuybrechsInstability}. 
Hence we focus on positive quadrature weights $a \geq 0$. 

Graphical designs with respect to the frequency order are function-agnostic in the sense that there is no particular function data used to define them. 
If there is a certain class of graph functions $f^1, f^2, \ldots ,f^T : [n] \to \RR$ with respect to which one seeks to sample, it is possible that the first $k$ eigenvectors are not the most important $k$ eigenvectors for these particular functions.
A worst-case scenario is that each function is orthogonal to $\phi_1, \phi_2, \ldots \phi_k$. 
Thus we make use of the following broader definition. 

\begin{definition}[\cite{BabeckiThomasGaleDuality,babeckiShiroma}]
Let $G = ([n],E,w)$ be a positively weighted, connected, simple graph with eigenvectors $\phi_1 ,\phi_2,\ldots, \phi_n$ ordered by frequency, and let $1\in J\subset[n]$. 
A \emph{$J$-graphical design} is a pair $(S,a)$ which averages $\phi_j$ for all $j\in J $ simultaneously: 
\begin{align*}
    \frac{1}{n} \sum_{i \in [n]} \phi_j(i) = \sum_{i \in S} a_i \phi_j(i)  \qquad \text{  for all  } j \in J .
\end{align*}
\end{definition}
Using GSP terminology, this says that the design averages the given class of bandlimited $|J|$-sparse signals.
The averaging condition in (\ref{eq: globally}) simplifies for eigenvectors of $L$.  We use $[2:n]$ to denote the subset $\{2,\ldots, n\}$.
\begin{proposition}[{\cite[Lemma 2.4]{babeckiShiroma}}]
\label{proposition: orthogonal}
If $\phi_1, \phi_2, \ldots, \phi_n$ are the eigenvectors of $L$ and $j \in [2:n]$, then the pair $(S,a)$ averages $\phi_j$ if and only if
\begin{equation}
\label{eq: orthogonal}
   \sum_{i \in S} a_i \phi_j(i) = 0 
\end{equation}
Up to scaling, a vector $a \geq 0$ averages $\phi_1$ if and only if $\ones_n^\top  a = 1$.
\end{proposition}
We require $1\in J$ to avoid the trivial solution $(\varnothing, \zeros_n)$, where $\zeros_n$ represents the all-zeros vector of length $n$.
Ideally, we would like a graphical design to average every function $f:[n] \to \RR$ exactly, but this is not possible with a proper subset of nodes.  
Recall that averaging a basis of $\RR^n$ is equivalent to averaging every function on $\RR^n$.

\begin{corollary}[{\cite[Lemma 2.5]{BabeckiThomasGaleDuality}}]
The pair $(S,a)$ averages every eigenvector of $L$ if and only if $S = [n]$ and $a = \ones_n/n$.  
\end{corollary}
\begin{proof}
Let $\phi_1 = \ones_n, \phi_2, \ldots ,\phi_n$ be an orthogonal eigenbasis of $\RR^n$ from $L$.  
By Proposition~\ref{proposition: orthogonal},  $a \in \RR^n$ averages $\phi_j$ if and only if\[0 = \sum_{s \in S} \phi_j(s) = a^\top \phi_j,\] which is to say that $a$ is orthogonal to $\phi_j$. If $(S,a)$ averages the entire eigenbasis, then $a \in \spanset\{ \phi_2,\ldots, \phi_n\}^\perp = \spanset\{\ones_n\}$.
The assumption that $(S,a)$ also averages $\phi_1 = \ones_n$ implies that $a = \ones_n/n$.
In the other direction $([n], \ones_n/n)$ averages every eigenvector of $L$ by definition. 
\end{proof}

In the rest of this note we consider the extent to which one can exactly average certain eigenvectors and approximately average the remaining eigenvectors.

\section{Graphical Design Optimization}
\label{sec: graphical design optimization}

We now introduce mathematical programming formulations for graphical design optimization.
Recall $\phi_1, \phi_2, \ldots, \phi_n$ are the eigenvectors of $L$ ordered by frequency and $J \subset [n]$ indexes the subset of eigenvectors we seek to average.

\subsection{Exact Averaging}
\label{sec: exact averaging}

We are interested in finding \emph{sparse} graphical designs $(S,a)$ in the sense that $|S| \leq k$ for any given $k \in [n]$.
By Proposition~\ref{proposition: orthogonal}, if $1 \in J$, the problem of finding a $J$-graphical design supported on at most $k$ nodes corresponds to finding a feasible solution to the following system.
\begin{subequations}
\label{eq: feasibility simplified}
\begin{align}
    && \sum_{i\in[n]} y_i &\leq k, \label{cnstr: size} &  \\
    &&  a_i &\leq y_i, & \forall i \in [n] 
 \label{cnstr: coupling}\\
    && \sum_{i\in[n]} a_i \phi_1(i) &= \frac{1}{n}\sum_{i\in[n]} \phi_1(i), \label{cnstr: phi_1} &  \\
    && \sum_{i\in[n]} a_i \phi_j(i) &= 0, & \forall j \in J \setminus \{1\}  \label{cnstr: phi_j} \\
    &&  a &\in [0,1]^n & \\
    &&  y &\in \{0,1\}^n. &
\end{align}
\end{subequations}
Here, $y$ is the indicator vector of the subset $S$, meaning $y_i = 1$ if and only if $i \in S$.
Constraint~\eqref{cnstr: size} ensures $|S| \leq k$.
Constraints~\eqref{cnstr: coupling} ensure $\supp(a) \subseteq S$.
Lastly, constraints~\eqref{cnstr: phi_1}-\eqref{cnstr: phi_j} ensure $(S,a)$ averages $J$.

We note that, if $k \geq |J|$, there always exists a feasible solution to Problem~\eqref{eq: feasibility simplified} (see~\cite[Lemma~2, Remark~5]{RekhaStefanRandomWalks}, \cite[Theorem~3.14]{BabeckiThomasGaleDuality}, \cite[Theorem~3.9]{babeckiShiroma}).
However, it follows directly from Babecki and Shiroma~\cite[Theorem~6.3]{babeckiShiroma}
that, if $k < |J|$, it is NP-complete to decide the feasibility of Problem~\eqref{eq: feasibility simplified}.

\subsection{Approximate Averaging}
\label{sec: approximate averaging}

We now build on Problem~\eqref{eq: feasibility simplified} to formulate an optimization problem that can be solved efficiently and combines exact and \emph{approximate} averaging.
Let $\overline{J} = [n] \setminus J$ index the subset of eigenvectors we seek to average approximately.

\subsubsection{Sparse Solutions via Linear Programming}
Given any $c \in \RR^n$ and $J \subseteq [n]$ with $1 \in J$, consider the following linear program.
\begin{subequations}
\label{eq: lp c}
\begin{align}
    \minimize  && \sum_{i \in [n]} c_i a_i & & \label{cnstr: lp obj} \\
    \text{s.t.} 
            && \sum_{i\in[n]} a_i \phi_1(i) &= \frac{1}{n}\sum_{i\in[n]} \phi_1(i), &  \label{cnstr: lp phi_1} \\
            && \sum_{i\in[n]} a_i \phi_j(i) &= 0, & \forall j \in J \setminus \{1\} \label{cnstr: lp phi_J} \\
            &&  a &\in \RR_{\geq 0}^n. &
\end{align}
\end{subequations}
Problem~\eqref{eq: lp c} is feasible since letting $a_i = 1/n$ for all $i \in [n]$ forms a feasible solution.
Steinerberger and Thomas~\cite[Lemma~2]{RekhaStefanRandomWalks} show that there exists a sparse feasible solution $a$ to Problem~\eqref{eq: lp c} with $|\text{supp}(a)| \leq |J|$.
We note that their result is similarly implied and implemented by the following \emph{rank lemma} (Lemma \ref{lemma: rank lemma}) about basic feasible solutions to linear programs, a connection which the second author now also points out in \cite[Remark 6.8]{babeckiShiroma}.
See~\cite[Lemma~2.1.4]{lau2011iterative} for a proof of the rank lemma.

\begin{lemma}[Rank Lemma]
\label{lemma: rank lemma}
Let $P = \{x \in \RR^n: Ax = b, x \geq 0 \}$ where $A \in \RR^{m \times n}$ and $b \in \RR^{m \times 1}$, and let $x^*$ be a basic feasible solution to $P$.
Then, $|\supp(x^*)| \leq \rank(A) \leq m$.
\end{lemma}

\begin{theorem}
\label{theorem: sparse}
Given any $c \in \RR^n$ and $J \in [n]$ with $1 \in J$, let $a^*$ be a basic optimal solution to Problem~\eqref{eq: lp c}.
Then, $|\supp(a^*)| \leq |J|$.
\end{theorem}
\begin{proof}
Constraint~\eqref{cnstr: lp phi_1} together with the non-negativity of $a$ imply Problem~\eqref{eq: lp c} is bounded, so it has an optimal solution.
Moreover, note that the matrix $U_J \in \RR^{|J| \times n}$ whose rows are $\phi_j^T$ for $j \in J$ encodes constraints~\eqref{cnstr: lp phi_1}-\eqref{cnstr: lp phi_J} and satisfies $\rank(U_J) = |J|$.
Then, the claim follows by Lemma~\ref{lemma: rank lemma}.
\end{proof}

\begin{corollary}
\label{corollary: sparse}
Given any $c \in \RR^n$ and $J \in [n]$ with $1 \in J$, let $a^*$ be a basic optimal solution to Problem~\eqref{eq: lp c}.
Then, if $k \geq |J|$, $a^*$ induces a feasible solution to Problem~\eqref{eq: feasibility simplified}.
\end{corollary}

In other words, not only a sparse feasible solution exists, but for \emph{any} linear objective function a sparse \emph{optimal} solution can be found efficiently via linear programming.

\subsubsection{Objective Functions}

In light of the sparsity guarantee of Theorem~\ref{theorem: sparse}, our goal is to design an objective function $c \in \RR^n$ such that a basic \emph{optimal} solution to Problem~\eqref{eq: lp c} given $c$ not only exactly averages $J$, but also approximately averages $\overline{J}$.
In this way, we promote numerical integration that is robust to functions not spanned by $J$.

\begin{remark}
In this work we focus on robustness.
However, we note that there may be other deployment goals that can be modeled by an appropriately-designed objective function.
For example, if the goal is to minimize communication costs between monitored nodes and some base station, $c_i$ may account for the distance between node $i$ and the station.
\end{remark}

Note that by Corollary~\ref{corollary: sparse}, the choice of $J$ must depend on the desired level of sparsity $k$, so we require $|J| \leq k$.
Here we consider two approaches:
\begin{enumerate}
    \item As described in Section~\ref{sec: laplacian matrix}, one natural choice is for $J$ to index the first $k$ eigenvectors by the frequency ordering.
    That is,
    \begin{equation}
    \label{eq: freq ordering}
        J = [k].
    \end{equation}
    \item A different approach assumes access to some statistic of the family of functions we seek to average.
    In particular, suppose we are given a sample\footnote{
    Consider settings in which data can be collected over the entire graph, possibly at a large cost, as part of a preliminary field study.
    Then, the purpose of the graphical design optimization might be the strategic placement of sensors for permanent real-time monitoring.
    } mean $\bar{f} \in \RR^n$.
    Then, a natural choice is for $J$ to index the first $k$ eigenvectors in decreasing order of the size of the projection to $\bar{f}$.
    That is, let $I = (i_1, i_2, \ldots, i_n)$ be the sequence of indices in $J$ sorted in decreasing order of $|\phi_j^T\bar{f}|$ for $j \in [n]$.
    Then,
    \begin{equation}
    \label{eq: fbar ordering}
        J = \{i_1, i_2, \ldots, i_k\}.
    \end{equation}

\end{enumerate}

Ideally, the objective function \eqref{cnstr: lp obj} approximately averages $\overline{J}$ in the sense that
\begin{equation}
\label{eq: J bar}
    \sum_{j \in \overline{J}} \left| \sum_{i \in [n]}\phi_j(i) \right|
\end{equation}
or some weighted version of it is minimized (recall $\sum_{i\in[n]} \phi_j(i) = 0$ for all $1 < j \leq n$ by Proposition~\ref{proposition: orthogonal}).
Unfortunately, the absolute value terms cannot be easily linearized without introducing auxiliary constraints that interfere with Theorem~\ref{theorem: sparse} and its sparsity guarantee.

As an alternative, we consider linear objective functions of the form $c \in \RR^n$ that serve as surrogates for \eqref{eq: J bar}.
These objective functions are based on the following upper bound on the absolute integration error given a generic function $f \in \RR^n$ (this result is similar to and closely follows techniques in Steinerberger and Thomas~\cite[Proposition~10]{RekhaStefanRandomWalks}).
\begin{proposition}
\label{proposition: parametric bound}
Let $\phi_1, \phi_2, \ldots, \phi_n$ be the eigenvectors of $L$, $a$ be a feasible solution to Problem~\eqref{eq: lp c} given $J \subseteq [n]$ with $1 \in J$, and $f \in \RR^n$.
Then,
\begin{equation*}
    \left|\frac{1}{n}\sum_{i \in [n]} f(i)  - \sum_{i \in [n]} a_if(i) \right| \leq \sum_{i\in[n]} a_i \left| \sum_{j \in \overline{J}} \phi_j(i) (\phi_j^Tf)\right|. 
\end{equation*}
\end{proposition}
\begin{proof}
First, note that after rescaling to obtain $\phi_1 = \frac{1}{\sqrt{n}}\ones$ we can write
\begin{align*}
    f 
    = \sum_{j \in [n]} (\phi_j^T f) \phi_j 
    = \left(\frac{1}{n}\sum_{i \in [n]} f(i) \right) \ones + \sum_{j=2}^n (\phi_j^T f) \phi_j.
\end{align*}
Then,
\begin{align*}
    a^Tf 
    &= a^T\left(\left(\frac{1}{n}\sum_{i \in [n]} f(i) \right) \ones + \sum_{j=2}^n (\phi_j^T f) \phi_j\right) 
    = \frac{1}{n}\sum_{i \in [n]} f(i) + \sum_{j \in \overline{J}} (\phi_j^T f) (\phi_j^T a),
\end{align*}
where the second equality holds since $\phi_1^T a = 1$ and $\phi_j^Ta = 0$ for all $j \in J \setminus \{1\}$.
Therefore, the absolute integration error is given by
\begin{align*}
    \left| \frac{1}{n}\sum_{i \in [n]} f(i)  - \sum_{i \in [n]} a_if(i) \right| = \left| \sum_{j \in \overline{J}} (\phi_j^T f) (\phi_j^T a) \right|.
\end{align*}
Upon interchanging the order of summation we obtain
\begin{align*}
    \left|\frac{1}{n}\sum_{i \in [n]} f(i)  - \sum_{i \in [n]} a_if(i) \right| 
    &= \left| \sum_{i\in[n]} a_i \left(\sum_{j \in \overline{J}} \phi_j(i) (\phi_j^T f)\right)\right| 
    \leq \sum_{i\in[n]} a_i \left| \sum_{j \in \overline{J}} \phi_j(i) (\phi_j^T f)\right|,
\end{align*}
where the inequality holds by the triangle inequality and since $a \geq 0$.
\end{proof}
The next result follows from an application of the Cauchy-Schwarz inequality.
\begin{corollary}
\label{corollary: nonparametric bound}
If furthermore $\sqrt{\sum_{j \in \overline{J}} (\phi_j^T f)^2} \leq 1$, then
\begin{equation*}
    \left|\frac{1}{n}\sum_{i \in [n]} f(i)  - \sum_{i \in [n]} a_if(i) \right| \leq \sum_{i\in[n]} a_i \sqrt{\sum_{j \in \overline{J}} |\phi_j(i)|^2}.
\end{equation*}
Note that this holds for any $f \in \RR^n$ up to scaling.
\end{corollary}

Intuitively, we would like to find a feasible solution $a$ to Problem~\eqref{eq: lp c} that makes the upper bounds in Proposition~\ref{proposition: parametric bound} and Corollary~\ref{corollary: nonparametric bound} as tight as possible.
This suggests two approaches for the design of a cost function $c \in \RR^n$:
\begin{enumerate}
    \item In the absence of access to some statistic of the family of functions we seek to average, based on Corollary~\ref{corollary: nonparametric bound} we set
    \begin{equation}
    \label{eq: c nonparametric}
        c_i = \sqrt{\sum_{j \in \overline{J}} |\phi_j(i)|^2}
    \end{equation}
    for all $i \in [n]$.
    Here we leverage the fact that any scaling constant appearing in Corollary~\ref{corollary: nonparametric bound} (and hence in~\eqref{eq: c nonparametric}) does not change the set of optimal solutions.
    \item  If given access to a sample mean $\bar{f} \in \RR^n$, based on Proposition~\ref{proposition: parametric bound}, we set
    \begin{equation}
    \label{eq: c parametric}
        c_i = \left| \sum_{j \in \overline{J}} \phi_j(i) (\phi_j^T \bar{f})\right|
    \end{equation}
    for all $i \in [n]$.
\end{enumerate}

\section{Computational Experiments}
\label{sec: computational experiments}

In this section, we showcase our approach to graphical design optimization using travel demand data from NYC.
We use the \texttt{osmnx} package of Boeing~\cite{boeing2017osmnx} to obtain a crowdsourced, simple undirected graph $G = (V,E)$ representing the Manhattan road network.
Roughly speaking, the nodes $V$ represent intersections and the edges $E$ represent street segments between pairs of intersections.
The graph has $|V| = 4,294$ nodes and $|E| = 7,497$ edges.
The edges are weighted by length in meters.
We use the \texttt{networkx} package~\cite{hagberg2008exploring}to compute the spectra of the combinatorial Laplacian $L$ and the commercial mathematical programming solver \texttt{Gurobi}~\cite{gurobi} to implement Problem~\eqref{eq: lp c}.

We use data retrieved from the NYC Taxi and Limousine Commission (TLC)~\cite{taxidata} to obtain a collection $f^1, f^2, \ldots, f^T \in \RR^n$ of functions encoding the number of yellow taxicabs hailed at each node in $G$ over a number of days.
We focus on trips that took place on the weekdays of June 2016, with a start time between $7-10$am (i.e., the morning commute).
This leads to a total of $T=29$ different functions, one for each day.
We match the starting point of each trip to the nearest node in $G$ using geographical (latitude and longitude) coordinates.
To implement~\eqref{eq: c parametric} we compute the sample mean travel demand $\bar{f} = \frac{1}{T}\sum_{t \in [T]} f^t$.

Figure~\ref{fig: manhattan} shows a visual representation of the input data on the Manhattan road network, together with graphical designs obtained through Problem~\eqref{eq: lp c}.
Here $k = 214$, corresponding to roughly $5\%$ of the total number of nodes.
At first glance, it might seem as if letting $J$ be given by~\eqref{eq: freq ordering} and $c$ be given by \eqref{eq: c nonparametric} leads to a sparser graphical design (Figure~\ref{fig: manhattan b}).
However, by Theorem~\ref{theorem: sparse}, each of these graphical designs actually has the same sparsity guarantee.
Therefore, letting $J$ be given by~\eqref{eq: fbar ordering} and $c$ be given by \eqref{eq: c nonparametric} leads to a graphical design that prioritizes nodes sustaining significant demand (Figure~\ref{fig: manhattan d}).

\begin{figure}[ht]
    \centering
    \begin{subfigure}[b]{0.24\linewidth}
        \centering
        \includegraphics[width=\linewidth]{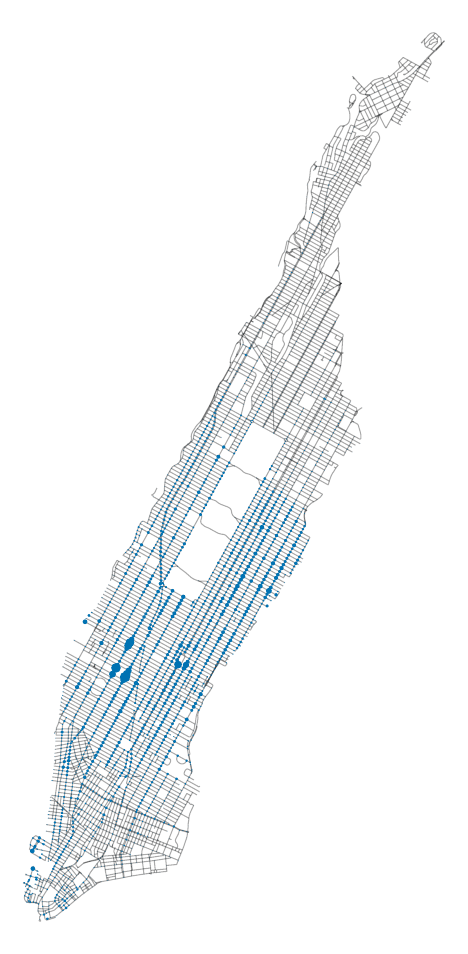}
        \caption{
        Input data.
        }
        \label{fig: manhattan a}
    \end{subfigure}
    \begin{subfigure}[b]{0.24\linewidth}
        \centering
        \includegraphics[width=\linewidth]{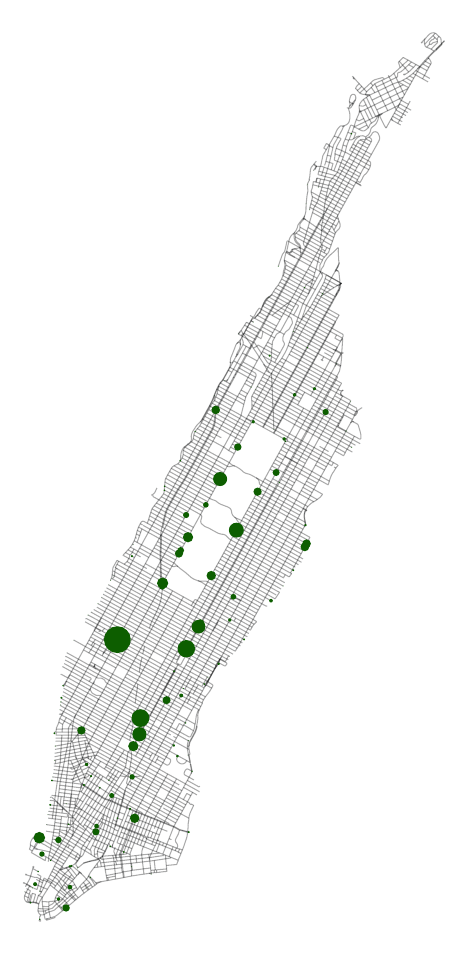}
        \caption{
        $J$ given by~\eqref{eq: freq ordering}, \\
        $c$ given by~\eqref{eq: c nonparametric}.
        }
        \label{fig: manhattan b}
    \end{subfigure}
    \begin{subfigure}[b]{0.24\linewidth}
        \centering
        \includegraphics[width=\linewidth]{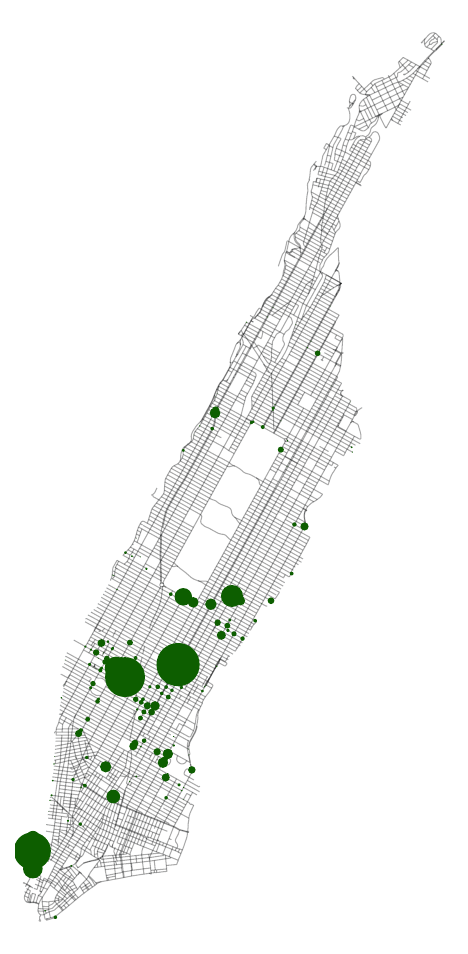}
        \caption{
        $J$ given by~\eqref{eq: fbar ordering}, \\
        $c$ given by~\eqref{eq: c nonparametric}.
        }
        \label{fig: manhattan c}
    \end{subfigure}
    \begin{subfigure}[b]{0.24\linewidth}
        \centering
        \includegraphics[width=\linewidth]{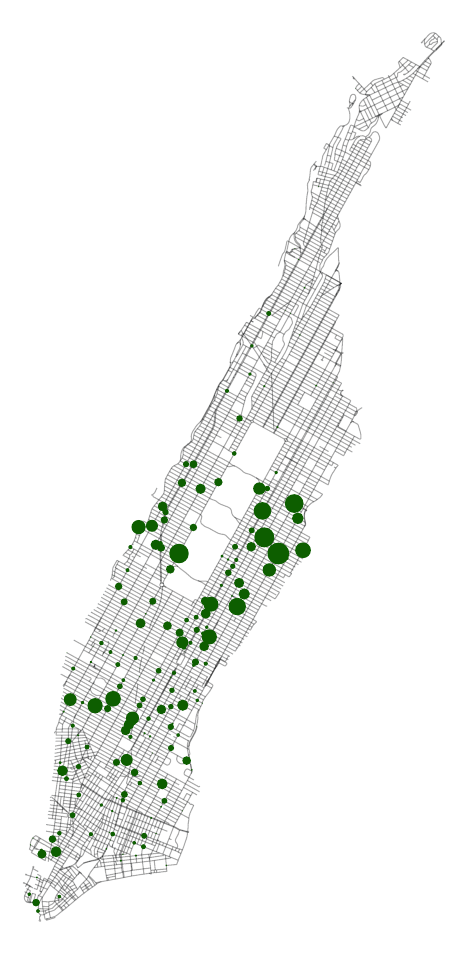}
        \caption{
        $J$ given by~\eqref{eq: fbar ordering}, \\ $c$ given by~\eqref{eq: c parametric}.
        }
        \label{fig: manhattan d}
    \end{subfigure}
    \caption{
    Graphical designs using up to $k=274$ nodes.
    The size of each node $i$ is proportional to its weight.
    Figure~\ref{fig: manhattan a}: Node $i$ is weighted by $\bar{f}(i)/n$.
    Figures~\ref{fig: manhattan b}-\ref{fig: manhattan d}: Node $i$ is weighted by $a_i^* \bar{f}(i)$, where $a^*$ is a basic optimal solution to Problem~\eqref{eq: lp c} with $J$ and $c$ as labeled in the corresponding figure.
    }
    \label{fig: manhattan}
\end{figure}

Figure~\ref{fig: validate} further shows the connection between graphical design accuracy and the choice of $J$ and $c$ in Problem~\eqref{eq: lp c}.
Given a graph function $f \in \RR^n$ and a graphical design $(S,a)$, the integration percent error is given by
\begin{equation}
\label{eq: percent error}
    \left|1 - \frac{\sum_{i \in S} a_i f(i)}{\frac{1}{n}\sum_{i \in [n]} f(i)} \right| \cdot 100.
\end{equation}
The figures show that the integration error tends to decrease as the sparsity parameter $k$ increases (i.e., as less sparse solutions are admitted).
However, access to a sample mean $\bar{f}$ enables high accuracy graphical design from early on (Figures~\ref{fig: validate b}-\ref{fig: validate c}).
In particular, if $\bar{f}$ is used to inform both the choice of both $J$ and $c$ in Problem~\eqref{eq: lp c}, a graphical design that uses under $5\%$ of the total number of nodes and consistently achieves under $5\%$ error is found.

\begin{figure}[ht]
    \centering
    \begin{subfigure}[b]{0.49\linewidth}
        \centering
        \includegraphics[width=\linewidth]{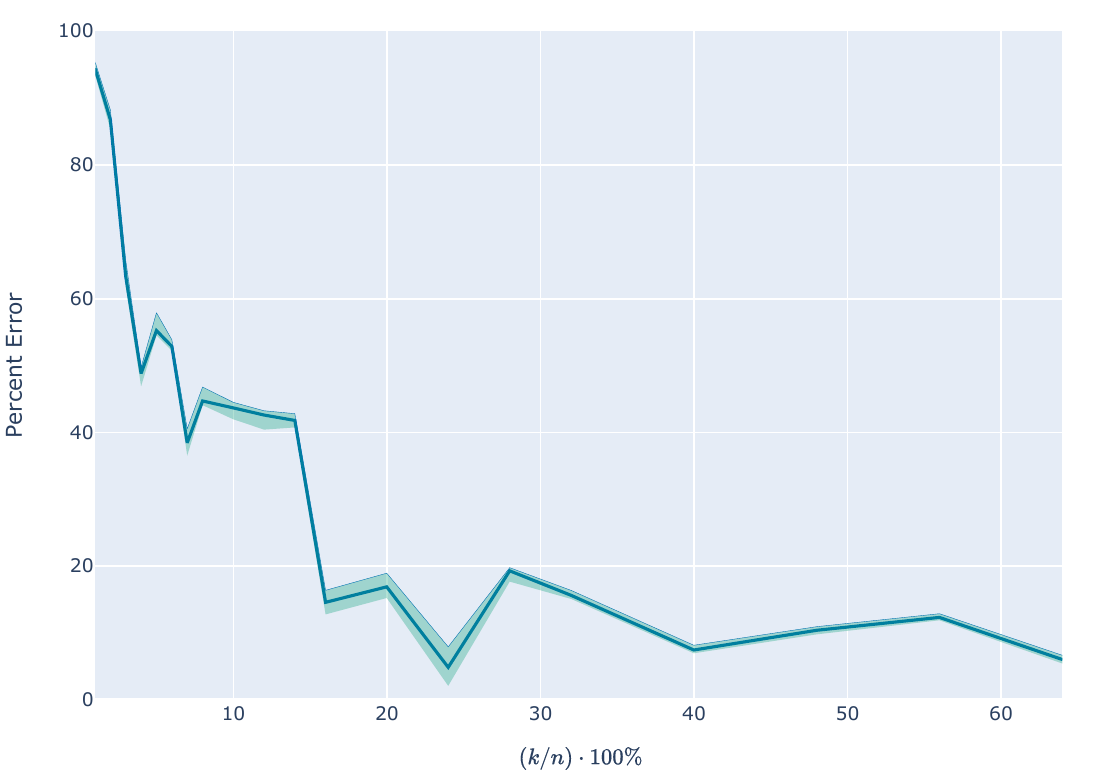}
        \caption{
        $J$ given by~\eqref{eq: freq ordering}, \\
        $c$ given by~\eqref{eq: c nonparametric}.
        }
        \label{fig: validate a}
    \end{subfigure}
    \begin{subfigure}[b]{0.49\linewidth}
        \centering
        \includegraphics[width=\linewidth]{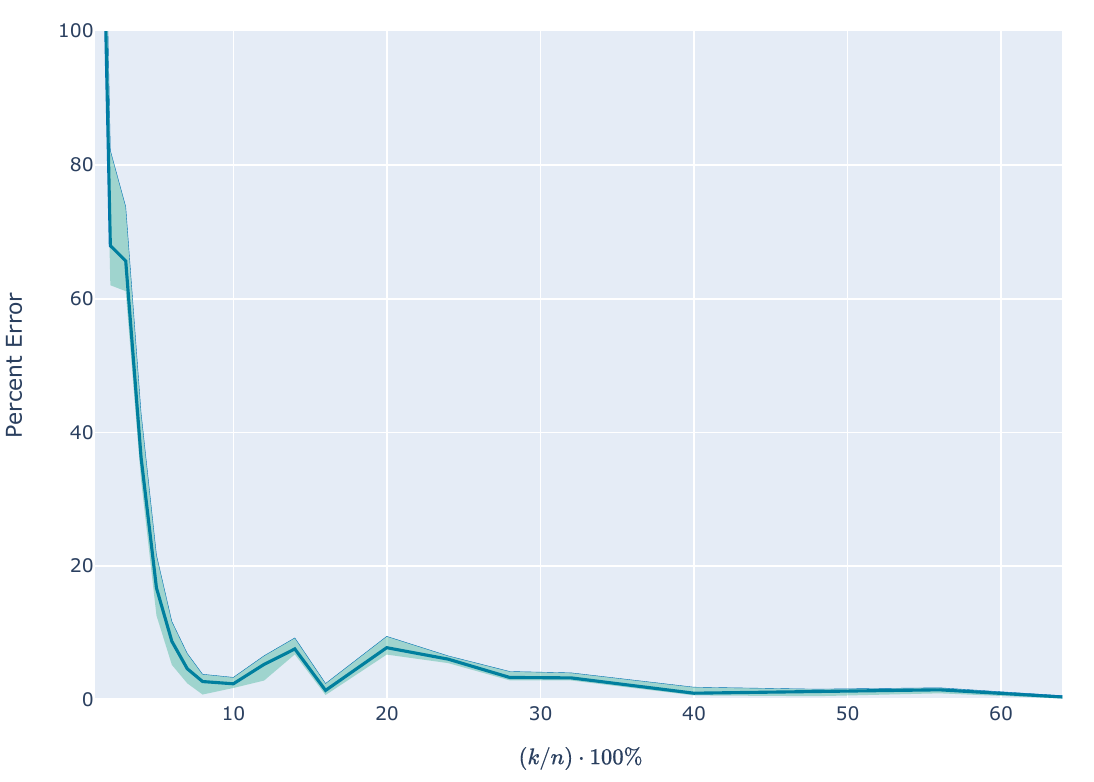}
        \caption{
        $J$ given by~\eqref{eq: fbar ordering}, \\
        $c$ given by~\eqref{eq: c nonparametric}.
        }
        \label{fig: validate b}
    \end{subfigure} \\ 
    \begin{subfigure}[b]{0.49\linewidth}
        \centering
        \includegraphics[width=\linewidth]{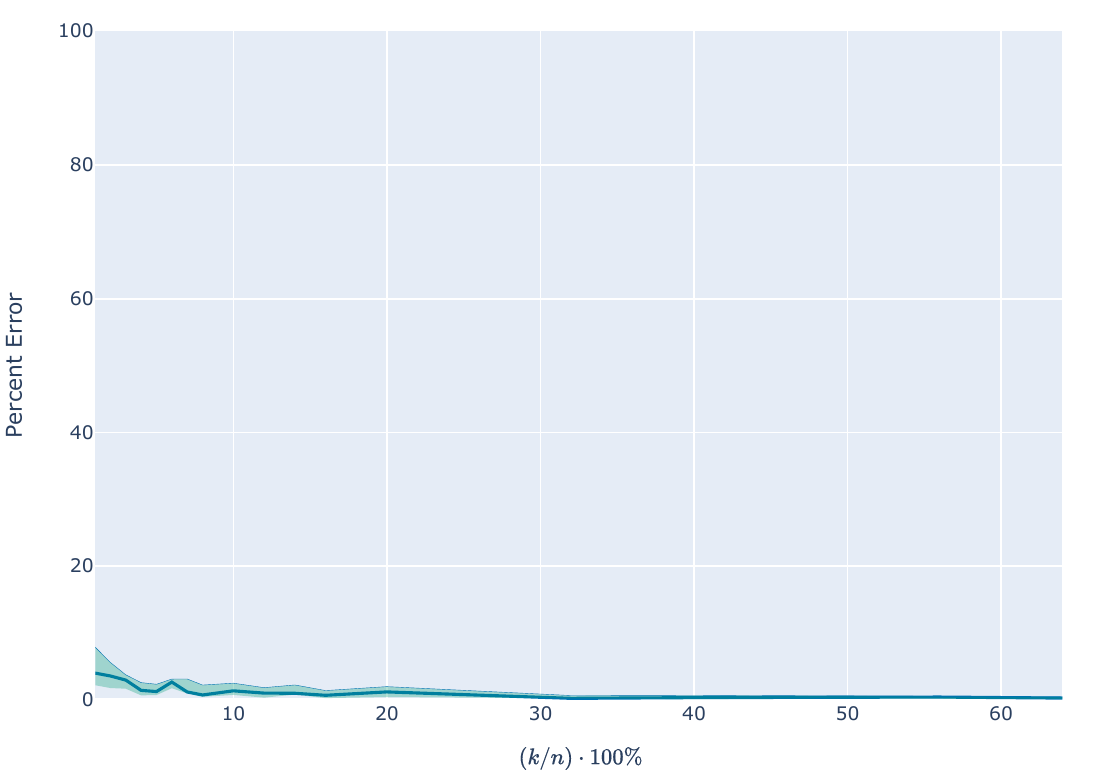}
        \caption{
        $J$ given by~\eqref{eq: fbar ordering}, \\
        $c$ given by~\eqref{eq: c parametric}.
        }
        \label{fig: validate c}
    \end{subfigure}
    \caption{
    Integration percent error as a function of sparsity parameter $k$ (shown as a percentage of the total number of nodes).
    In each panel, the vertical spread is over the $29$ functions $f^1, f^2, \ldots, f^{29}$ collected from the TLC data set: the solid line is corresponds to the median and the shaded region corresponds to the interquartile range.
    Figures~\ref{fig: validate a}-\ref{fig: validate c}:
    The percent error is computed using \eqref{eq: percent error} and a basic optimal solution $a^*$ to Problem~\eqref{eq: lp c} with $J$ and $c$ as labeled in the corresponding figure.
    }
    \label{fig: validate}
\end{figure}

\begin{remark}
Setting $c = \ones$ to be the all-ones vector reduces Problem~\eqref{eq: lp c} to finding any basic feasible solution (this follows from Constraint~\eqref{cnstr: lp phi_1}).
We report that, surprisingly, this approach leads to graphical designs whose performance tends to improve on Figure~\ref{fig: validate a} (but not on Figure~\ref{fig: validate b}-\ref{fig: validate c}).
In other words, in this particular implementation, the first basic feasible solution obtained by the solver tends to induce a graphical design of fair accuracy.
However, in principle, this approach could equally induce low accuracy graphical designs, given that they are interchangeable whenever $c$ is the all-ones vector.
\end{remark}

\section*{Acknowledgements}
This material is based upon work supported by the National Science Foundation under Grant No. DMS-1929284 while the authors were in residence at the Institute for Computational and Experimental Research in Mathematics in Providence, RI, during the \emph{Discrete Optimization: Mathematics,  Algorithms, and Computation} semester program.
The work of H. Al-Thani was made possible by the Graduate Sponsorship Research Award from the Qatar National Research Fund (a member of Qatar Foundation). The findings herein reflect the work, and are solely the responsibility, of the authors. 
J.~C. Mart\'inez Mori was partially supported by NSF Grant No.
2144127, awarded to S. Samaranayake. 
J.~C. Mart\'inez Mori is supported by Schmidt Science
Fellows, in partnership with the Rhodes Trust.

\printbibliography

\end{document}